\documentclass{amsart}

\usepackage{amsmath}%
\usepackage{amsfonts}%
\usepackage{amssymb}%
\usepackage{amsthm}
\usepackage{graphicx}
\usepackage{dsfont}
\usepackage{pdfcomment}
\usepackage{tikz}
\usepackage{tikz-3dplot}
\usepackage{subfigure}
\usepackage{soul}
\usepackage{enumerate}
\usepackage{cancel}
\usepackage{stmaryrd}
\usepackage{mathtools}

\usepackage{hyperref}
\usepackage{cleveref}
\usepackage{esvect}

\usetikzlibrary{decorations.markings}
\newtheorem{theorem}{Theorem}[section]

\newtheorem{corollary}[theorem]{Corollary}

\newtheorem{definition}[theorem]{Definition}

\allowdisplaybreaks

\newtheorem{lemma}[theorem]{Lemma}

\newtheorem{proposition}[theorem]{Proposition}
\newtheorem{remark}[theorem]{Remark}

\newtheorem{thmx}{Conjecture}

\DeclareMathOperator{\Ad}{Ad}
\DeclareMathOperator{\Arg}{Arg}
\DeclareMathOperator{\tr}{tr}

\DeclareMathOperator{\Ker}{Ker}

\DeclareMathOperator{\Real}{Re}
\DeclareMathOperator{\Imaginary}{Im}

\DeclareMathOperator{\hol}{hol}

\DeclareMathOperator{\Hom}{Hom}
\DeclareMathOperator{\Teich}{Teich}
\renewcommand{\Re}{\Real}
\renewcommand{\Im}{\Imaginary}

\title{Pullback of symplectic forms to the space of circle patterns}
\author{Wai Yeung Lam}
\thanks{This work was partially supported by the FNR grant CoSH O20/14766753.}
\address{Department of Mathematics, University of Luxembourg, Maison du nombre, 6 avenue de la Fonte, L-4364 Esch-sur-Alzette, Luxembourg.} \email{wyeunglam@gmail.com}

\begin{document}

	\begin{abstract}
		We consider circle patterns on surfaces with complex projective structures. We investigate two symplectic forms pulled back to the deformation space of circle patterns. The first one is Goldman's symplectic form on the space of complex projective structures on closed surfaces. The other is the Weil-Petersson symplectic form on the Teichm\"{u}ller space of punctured surfaces. We show that their pullbacks to the space of circle patterns coincide. It is applied to prove the smoothness of the deformation space, which is an essential step to the conjecture that the space of circle patterns is homeomorphic to the Teichm\"{u}ller space of the closed surface. We further conjecture that the pullback of the symplectic forms is non-degenerate and defines a symplectic structure on the space of circle patterns.
	\end{abstract}
	
		\maketitle

\section{Introduction}

Conformal maps between surfaces are fundamental in low dimensional topology with a wide range of applications. They are characterized as mappings that locally preserve angles, sending infinitesimal circles to themselves. Instead of infinitesimal size, a circle packing in the plane is a configuration of finite-size circles where certain pairs are mutually tangent. Thurston proposed regarding the map induced from two circle packings with the same tangency pattern as a discrete conformal map \cite{Stephenson2005}. He realized that circle packings provide a natural discretization of the Riemann mapping theorem, which is verified by Rodin and Sullivan \cite{Rodin1987}. With circle packings, one could define discrete conformal structures and compare them with classical conformal structures.

Generalizing the notion of circle packings, we consider configurations of circles which are allowed to intersect. A circle pattern in the plane is a realization of a planar graph such that each face has a circumcircle passing through the vertices. By adding diagonals, we assume each face is a triangle and the circle pattern is determined by cross ratios as follows. Assume $(V, E, F)$ is a triangulation of the planar graph. Given a realization of the vertices $z:V \to \mathbb{C}\cup\{\infty\}$, we associate a \textit{complex} cross ratio to the common edge $\{ij\}$ shared by triangles $\{ijk\}$ and $\{jil\}$ (See Figure \ref{fig:delaunay}):
\[
X_{ij} :=  -\frac{(z_k - z_i)(z_l -z_j)}{(z_i - z_l)(z_j - z_k)} =X_{ji} \in \mathbb{C}
\]
\begin{figure}[h!] \centering
				\includegraphics[width=0.85\textwidth]{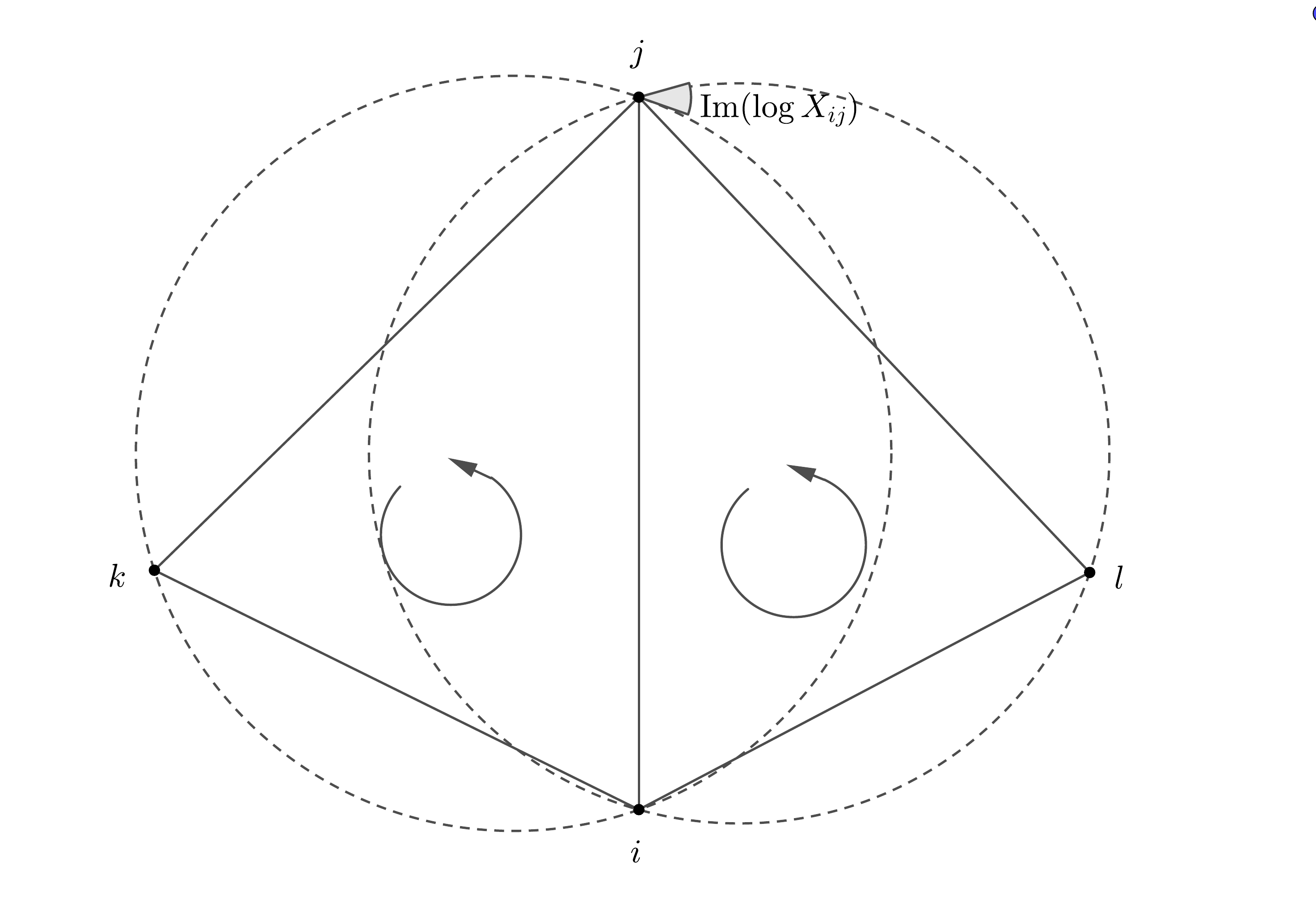}
	\caption{Two triangular faces $ijk$ and $jil$ share a common edge $ij$.} 
	\label{fig:delaunay}
\end{figure}
It defines a function  $X: E  \to \mathbb{C}$ satisfying certain polynomial equations, which can be defined generally on any surface. Throughout the paper, we consider closed oriented surfaces $S_g$ with genus $g>1$.
\begin{definition}\label{def:crsys}
	Suppose $(V,E,F)$ is a triangulation of a closed surface $S_{g}$, where $V$, $E$ and $F$ denote respectively the sets of vertices, edges and faces. An unoriented edge is denoted by $\{ij\}=\{ji\}$ indicating that its end points are vertices $i$ and $ j$. A cross ratio system is an assignment $X:E \to \mathbb{C}$ such that for every vertex $i$ with adjacent vertices numbered as $1$, $2$, ..., $r$ in the clockwise order counted from the link of $i$,
	\begin{gather} 
		\Pi_{j=1}^n X_{ij} =1  \label{eq:crproduct}\\
		X_{i1} + X_{i1} X_{i2} + X_{i1}X_{i2}X_{i3} + \dots +  X_{i1}X_{i2}\dots X_{ir} =0 \label{eq:crsum}
	\end{gather}
	where $X_{ij} = X_{ji}$. 
\end{definition}

A cross ratio system provides a recipe to lay out neighboring circumdisks. Equations \eqref{eq:crproduct} and \eqref{eq:crsum} ensure that the holonomy around each vertex under the gluing construction is the identity. Indeed, the holonomy around vertices are M\"{o}bius transformations generally, which could be represented as matrices in $SL(2,\mathbb{C})$. Equation \eqref{eq:crproduct} ensures the corresponding matrix to have eigenvalues $1$ while Equation \eqref{eq:crsum} demands the off-diagonal entries to be zero.

We are interested in circle patterns with prescribed Delaunay intersection angles. Observe that the imaginary part $\Im (\log X) : E \to [0,2\pi)$ is the intersection angles of circumdisks. 

\begin{definition} Given a triangulation of $S_g$, a Delaunay angle structure is an assignment of angles $\Theta:E \to [0,\pi)$ with $\Theta_{ij}=\Theta_{ji}$ satisfying the following:   
	\begin{enumerate}[(i)]
		\item For every vertex $i$,  \[ \sum_j \Theta_{ij} = 2\pi\] where the sum is taken over the neighboring vertices of $i$ on the universal cover.
		\item For any collection of edges $(e_0,e_1,e_2,\dots,e_r=e_0)$ whose dual edges form a simple closed contractable path on the surface, then
		\[
		\sum_{i=1}^{n} \Theta_{e_i} > 2\pi
		\]
		unless the path encloses exactly one primal vertex.
	\end{enumerate}
	We consider cross ratio systems with prescribed Delaunay angles 
	\[
	P(\Theta) = \{ X: E \to \mathbb{C}| \Im (\log X) = \Theta \text{ and satisfying } \eqref{eq:crproduct}\eqref{eq:crsum} \}.
	\]
\end{definition}    
It is known that these conditions on $\Theta$ are the necessary and sufficient for $P(\Theta)$ to be non-empty. Particularly, for any Delaunay angle structure $\Theta$, there is a unique Fuchsian structure that supports a circle pattern with intersection angle $\Theta$ (See \cite{Bobenko2004}).

Generally, every Delaunay cross ratio system in $P(\Theta)$ induces a complex projective structure on the closed surface. We denote $P(S_g)$ the space of marked complex projective structures on the closed surface and it is known that $P(S_g)\cong \mathbb{R}^{12g-12}$. By forgetting the circle pattern and considering the underlying complex projective structure, one obtains a mapping from the space of circle patterns to the Teichm\"{u}ller space of the closed surface via a composition of maps
\[
P(\Theta) \xrightarrow{f} P(S_g) \xrightarrow{\pi} \Teich(S_g) \cong \mathbb{R}^{6g-6}
\]
where $f$ is the forgetful map and $\pi$ is the classical uniformization map for Riemann surfaces.

\begin{thmx}\cite{KMT2003} \label{con:kmt}
	For any Delaunay angle structure $\Theta:E \to [0,\pi)$, the map \[\pi \circ f : P(\Theta) \to \Teich(S_g)\] is a homeomorphism. In particular, $P(\Theta)$ is a smooth manifold of real dimension $6g-6$ and the mapping $f:P(\Theta)  \to    P(S_g)$ is an embedding. In other words, $f(P(\Theta))$ is a section of the fiber bundle $P(S_g) \xrightarrow{\pi} \Teich(S_g)$.
\end{thmx}

The conjecture suggests an interesting difference between discrete conformal structures and classical conformal structures. Elements in $P(\Theta)$ for a fixed $\Theta$ are regarded as sharing the same discrete conformal structure. If discrete conformal structures corresponded exactly to classical conformal structures, then $f(P(\Theta))$ would lie within a single fiber. Instead, the conjecture suggests that $f(P(\Theta))$ is transversal to the fibers everywhere. The conjecture is known to hold for tori (genus $g=1$) \cite{Lam2021}. When $g>1$, it is still largely open. Although there are partial results \cite{Schlenker2018,BW2023}, it is not yet clear whether $P(\Theta)$ is smooth. 

This paper investigates the pullback of two symplectic forms to $P(\Theta)$. Firstly, over the space of complex projective structures $P(S_g)$, there is a complex symplectic form $\omega_G$ and Goldman showed that it coincides with the Weil-Petersson symplectic form over the section of the Fuchsian structures \cite{Goldman1984}. By forgetting the circle patterns
\[
 f: P(\Theta) \to P(S_g) \cong \mathbb{R}^{12g-12},
\]
 the symplectic form $\omega_G$ is pulled back to the space of circle patterns $P(\Theta)$. Secondly, there is an embedding of $P(\Theta)$ to the Teichm\"{u}ller space of a punctured surface $\Teich(S_{g,n})$ with $n=|V|$, which consists of marked complete hyperbolic metrics with cusps. By viewing the Riemann sphere as the boundary of hyperbolic 3-space, every Delaunay circle pattern corresponds to a locally convex pleated surface in $\mathbb{H}^3$. The complex cross ratios $X$ provides a recipe to construct the pleated surface by gluing ideal triangles, with shear coordinates $\Re \log X$ and bending angles $\Im \log X$. By forgetting the bending angles, every circle pattern induces a complete hyperbolic metric with cusps on the punctured surface. Thus we have an embedding
\[
\tilde{f}:P(\Theta)  \xhookrightarrow{} \Teich(S_{g,n}) \cong \mathbb{R}^{6g-6+2n}
\]
Via the embedding $\tilde{f}$, the Weil-Petersson symplectic form $\omega_{P}$ on $\Teich(S_{g,n}) $ is pulled back to $P(\Theta)$. We shall show that the pullbacks of $\omega_G$ and $\omega_P$ coincide on $P(\Theta)$.

To state our results precisely, for any Delaunay $X \in P(\Theta)$, we consider a complex vector space  $W_{X}^{\mathbb{C}}$  defined by the linearization of the cross-ratio equations \eqref{eq:crproduct}\eqref{eq:crsum} (See Definition \ref{def:wc}). If $P(\Theta)$ is smooth, then $W_{X}^{\mathbb{R}}:= W_{X}^{\mathbb{C}} \cap \mathbb{R}^{E}$ is isomorphic to the tangent space $T_X P(\Theta)$.
\begin{theorem}\label{thm:main}
	The pullback of Goldman's symplectic form $\omega_G$ and the complexfied symplectic form $\omega^{\mathbb{C}}_P$ on the complex vector space $W_{X}^{\mathbb{C}}$ satisfy
	\[
	\omega_G= \frac{1}{2} \omega_P^{\mathbb{C}}.
	\]
	It implies that the image of $W_X^{\mathbb{R}}$ under the holonomy map
	\[
		\hol: W_X^{\mathbb{R}} \to H^1_{\Ad\rho}(\pi_1(S_g),sl(2,\mathbb{C})) \cong T_{f(X)}P(S_g)
	\]
	is isotropic with respect to the imaginary part of the symplectic form $\Im \omega_{G}$ and hence has real dimension at most $6g-6$, which is half of $\dim_{\mathbb{R}} T_{f(X)}P(S_g)$.
\end{theorem}

We apply it to study the smoothness of $P(\Theta)$.

\begin{corollary}\label{cor:infimplysmooth}
	If a circle pattern $X \in P(\Theta)$ is infinitesimally rigid (i.e. $\Ker \hol$ is trivial), then
	\[
	\dim_{\mathbb{R}}  W_X^{\mathbb{R}} =6g-6.
	\]
	It yields that in a neighborhood of $X$, $P(\Theta)$ is a real analytic manifold of dimension $6g-6$ and under the forgetful map, $f(P(\Theta))$ is a Lagrangian submanifold of $P(S_g)$ with respect to the real symplectic form $\Im \omega_G$. 
\end{corollary}

Our results reduce the problem of smoothness to infinitesimal rigidity. There have been partial results on infinitesimal rigidity by various means. For circle patterns at Fuchsian structures with arbitrary Delaunay angles $\Theta$, infinitesimal rigidity can be proved by considering the discrete hyperbolic Laplacian \cite{Lam2024I}. In the setting of circle packings, Bonsante and Wolf \cite{BW2023} proved infinitesimal rigidity for general complex projective structures by an index argument, where they also verify the smoothness. Here circle packings are special cases of Delaunay circle patterns, since a circle packing superposed with its dual packing is a Delaunay circle pattern with $\Theta$ taking values either $0$ or $\pi/2$ (See Figure \ref{fig:circles}).
 
If Conjecture \ref{con:kmt} holds, then $P(\Theta)$ is a manifold of even dimension and it makes sense to ask whether the pullback of the symplectic form defines a symplectic structure on $P(\Theta)$. 
\begin{thmx}\label{conj:strong}
	The bilinear form $\Re \omega_{G} = \frac{1}{2} \omega_{P}$ over $P(\Theta)$ is non-degenerate. 
\end{thmx}

In the last section, we collect several observations of the symplectic forms. In a forthcoming paper, we verify Conjecture \ref{conj:strong} in the case of tori using the discrete Laplacian.

\begin{figure}
		\includegraphics[width=0.7\textwidth]{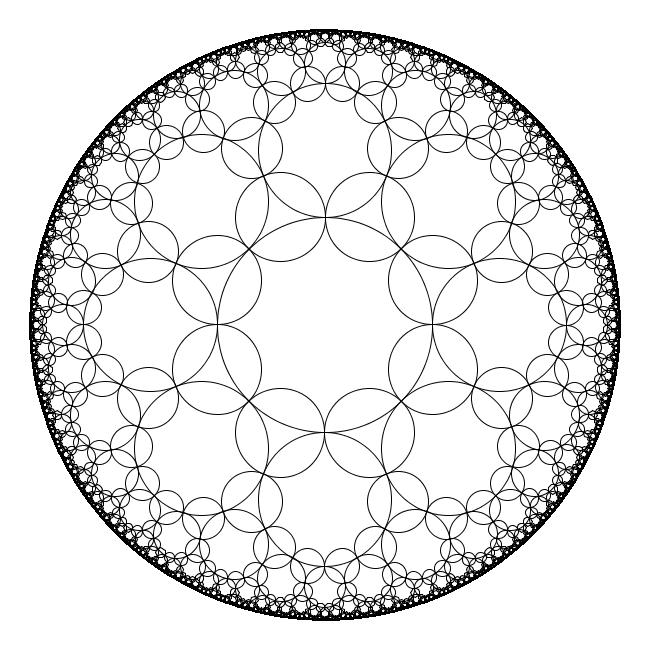}
\caption{A circle pattern at a Fuchsian structure obtained by superposing a circle packing with its dual packing. In this case, the intersection angles $\Im (\log X)$ are all $\pi/2$. The intersection points form our vertex set $V$. Two vertices are joined by an edge if the edge forms a common chord of two circles. In this way, we obtain a Delaunay decomposition with a bunch of triangular faces and octagonal faces. In order to parameterize the circle pattern by cross ratios, we triangulate faces by adding diagonals. Over these extra diagonals, the intersection angles $\Im (\log X)$ are zero since neighbouring circumcircles coincide.} 
\label{fig:circles}
\end{figure}

\subsection{Related work}

Conjecture \ref{con:kmt} and the symplectic form have their correspondence in Thurston's grafting construction. Thurston introduced a grafting construction to produce a complex projective structure from a bending lamination on a closed hyperbolic surface and suggest that the grafting map
\[
\mbox{Gr}: \Teich(S_g) \times \mbox{ML}(S_g)\to P(S_g)
\]
is a homeomorphism, where $\mbox{ML}(S_g)$ is the set of bending measured laminations \cite{Tan1992}. Scannell-Wolf \cite{Scannell2002} showed that for a fixed bending measured lamination $\Theta$, the conformal grafting
\[
\mbox{gr}_{\Theta}:  \Teich(S_g)  \to  \Teich(S_g)
\]
defined via compositions
\[
 \Teich(S_g) \rightarrow  \Teich(S_g) \times \{\Theta\} \xrightarrow{\mbox{Gr}} P(S_g) \xrightarrow{\pi} \Teich(S_g) 
\]
is a homeomorphism. Regarding symplectic structures, there is a symplectic form $\omega_H$ over $\Teich(S_g) \times \mbox{ML}(S_g)$ induced from the cotangent bundle while there is Goldman's symplectic form $\omega_G$ on $P(S_g)$. By considering the re-normalized volume of the hyperbolic ends in hyperbolic 3-space, Krasnov-Schlenker \cite{Krasnov2009} showed that the pullback of the real part $\Re \omega_{G}$ under the grafting map $\mbox{Gr}$ coincides with $\omega_H$. Particularly, it is non-degenerate. 

In the setting of circle patterns, we fix a Delaunay angle structure $\Theta$ over a triangulation of a closed surface. The triangulation could be regarded as an ideal triangulation of a punctured surface $S_g - V$ with $\Theta$ being a bending measured lamination. For any complete hyperbolic metric with cusps on the punctured surface, one can perform grafting to obtain a complex projective structure which might or might not have singularities at the punctures. The set $\tilde{f}(P(\Theta)) \subset \Teich(S_{g,n})$ is the subset which yields complex projective structures \textit{without} singularities at the punctures under the grafting construction. Conjecture \ref{con:kmt} is analogous to Scannell-Wolf's results asking whether the composition
\[
\tilde{f}(P(\Theta)) \subset \Teich(S_{g,n}) \rightarrow  \tilde{f}(P(\Theta)) \times \{\Theta\} \xrightarrow{\mbox{Gr}} P(S_g) \xrightarrow{\pi} \Teich(S_g) 
\]
is a homeomorphism. In this context, our theorem relating the symplectic form $\omega_P$ on $\Teich(S_{g,n})$ to $\omega_G$ on $P(S_g)$ is analogous to Krasnov-Schlenker's results. It would be interesting to understand how to deduce our results by fully adopting Krasnov-Schlenker's approach. Since both the real part and the imaginary part of $\omega_G$ are involved in our case, it is expected to consider the complex volume of a hyperbolic end. The real part is the re-normalized volume and the imaginary part should be the Chern-Simons invariant, whose meaning in our context is however superficial. Instead, our approach is closer to Bonahon-S\"{o}zen's approach \cite{BonSo2001}

Definition of \ref{def:crsys} can be seen as the gluing equation for a pleated surface from ideal triangles in hyperbolic 3-space. There is an analogous gluing equation for hyperbolic 3-manifolds from ideal tetrahedra \cite{Neumann1985}. In their study, a symplectic relation is used to deduce the smoothness of the configuration space \cite{Choi2004}. It is analogous to our use of the symplectic form to study the smoothness of $P(\Theta)$. 

There is vast literature studying complex projective structures on punctured surfaces, which are related to meomorphic quadratic differentials \cite{Guruprasad1998}. However, the complex projective structures with trivial holonomy at punctures as in our case are usually excluded.

\section{Linearization of cross-ratio equations}

The set $P(\Theta)$ is defined as the zero set of the algebraic system \eqref{eq:crproduct}\eqref{eq:crsum}. Its smoothness amounts to showing that the Jacobian of the algebraic system has constant nullity and then the smoothness follows from the constant rank theorem. It motivates to study the linearization of \eqref{eq:crproduct}\eqref{eq:crsum} and identify its kernel using logarithmic derivative.

 Suppose $X^{(t)}:E \to \mathbb{C}$ is a 1-parameter family of cross ratios satisfying Equation \eqref{eq:crproduct} \eqref{eq:crsum} with $X=X^{(t)}|_{t=0}$, then by differentiation the logarithmic derivative $x:= \frac{d}{dt}(\log X^{(t)})|_{t=0}$ satisfies for every vertex $i$ with adjacent vertices numbered as $1$, $2$, ..., $r$ in the clockwise order counted from the link of $i$,
 \begin{gather} 
 	\sum_{j=1}^r x_{ij} =0  \label{eq:xcrproduct}\\
 	x_{i1} X_{i1} + (x_{i1}+x_{i2}) X_{i1} X_{i2} + \dots +  (\sum_{j=1}^n x_{ij}) X_{i1}X_{i2}\dots X_{ir} =0 \label{eq:xcrsum}
 \end{gather}
 Furthermore, if $\Arg X^{(t)}$ remains constant for all $t$, then $x$ is purely real.

\begin{definition}\label{def:wc}
	Given a cross ratio system $X\in P(\Theta)$, we define a complex vector space $W_{X}^{\mathbb{C}}$ as the collection of functions $x:E \to \mathbb{C}$ satisfying Equation \eqref{eq:xcrproduct} \eqref{eq:xcrsum} for every vertex $i$.
	We further define a real vector space
	\[
	W_{X}^{\mathbb{R}}:= W_{X}^{\mathbb{C}} \cap  \mathbb{R}^{E}.
	\]
\end{definition}

Recall from the Euler characteristic, we have \[|E|=6g-6+3|V|=6g-6+3n.\]
By counting the number of variables and constraints, one deduces that
\begin{gather}
	\dim_{\mathbb{C}} W_{X}^{\mathbb{C}} \geq |E| - 2|V| = 6g-6 +|V| = 6g-6 + n.  \label{eq:VCdim}\\
	\dim_{\mathbb{R}} W_{X}^{\mathbb{R}} \geq |E| - 3|V| = 6g-6.  \label{eq:VRdim}
\end{gather}
where $n:=|V|$. The equalities hold if and only if the constraints are linearly independent. We collect some direct observations here.

\begin{proposition}
	Suppose $X \in P(\Theta)$ is a Delaunay circle pattern. The complex vector space $W_{X}^{\mathbb{C}}$ is identified as the kernel of the Jacobian of the algebraic system \eqref{eq:crproduct}\eqref{eq:crsum} via logarithmic derivative. Furthermore
	\[
	\dim_{\mathbb{C}} W_{X}^{\mathbb{C}} = 6g-6 + n
	\]
	and $W_{X}^{\mathbb{C}}\cong T_X (P(S_{g}) \times \mbox{Conf}_n(S_g))$ where  $\mbox{Conf}_n(S_g)$ is the set of n-tuples of pairwise distinct points on $S_g$. 
\end{proposition}
\begin{proof}
	By forgetting the circles, every Delaunay circle pattern yields a complex projective structure on $S_g$ together with marked points $V$. Then every element in $W_{X}^{\mathbb{C}}$ corresponds to an infinitesimal deformation of marked points $V$ as well as the projective structure. This mapping is injective and thus $\dim_{\mathbb{C}} W_{X}^{\mathbb{C}}\leq 6g-6 + n$. Together with Equation \eqref{eq:VCdim}, we obtain the claim.
\end{proof}

In contrast, it is not obvious how to determine the dimension of $W_{X}^{\mathbb{R}}$. Deducing $\dim_{\mathbb{R}} W_{X}^{\mathbb{R}}$ is the heart of the question.

\begin{proposition}
	 The real vector space $W_{X}^{\mathbb{R}}$ corresponds to infinitesimal deformations of cross ratios that preserve $\Im \log X = \Theta$. They are identified as infinitesimal deformation of circle patterns with fixed intersection angles. If $\dim_{\mathbb{R}} W_{X}^{\mathbb{R}}=6g-6$ for all $X \in P(\Theta)$, then $P(\Theta)$ is a real analytic manifold of dimension $6g-6$ and 
	\[
	T_{X} P(\Theta) \cong W_{X}^{\mathbb{R}}.
	\] 
\end{proposition}

We shall make use of symplectic forms to study $\dim W_{X}^{\mathbb{R}}$. Indeed, as explained in Section \ref{sec:symcom}, Equations \eqref{eq:xcrproduct} and \eqref{eq:xcrsum} represent certain symplectic relations.

\section{Teichm\"{u}ller space $\Teich(S_{g,n})$ and Weil-Petersson symplectic form}

We recall the theory of decorated Teichm\"{u}ller space of punctured surfaces \cite{Penner2012}. Given a triangulation $(V,E,F)$ of a closed surface $S_g$, we can interpret it as an ideal triangulation of a punctured surface $S_{g,n}:=S_g-V$ where $n=|V|$.

We consider the Teichm\"{u}ller space $\Teich(S_{g,n})$ consisting of marked complete hyperbolic metrics with cusps at the $n$ punctures. It is known that
\[
\Teich(S_{g,n}) \cong  \mathbb{R}^{6g-6+2n}.
\]
With the triangulation fixed, the Teichm\"{u}ller space can be parametrized by \emph{positively real} cross ratios $X:E \to \mathbb{R}_{>0}$ on edges as follows. Given a complete hyperbolic metric with cusps, there is a developing map of the universal cover to the hyperbolic plane in the Poincare disk model. For any edge $ij$, we focus on one of its lifts to the universal cover, with neighbouring triangles $ijk$ and $jil$. Via the developing map, the four corners are mapped to $z_i,z_j,z_k,z_l \in S^{1} \subset \mathbb{C}$. The cross ratio 
\[
X_{ij} :=  -\frac{(z_k - z_i)(z_l -z_j)}{(z_i - z_l)(z_j - z_k)}  \in \mathbb{R}_{>0}
\]
satisfying $X_{ij}=X_{ji}$ is independent of the choice of the lift. Hence it defines $X:E \to \mathbb{R}_{>0}$. Furthermore, we have for every vertex $i$ with adjacent vertices numbered as $1$, $2$, ..., $n$ in the clockwise order counted from the link of $i$,
\begin{gather}
	\Pi_{j=1}^n X_{ij} =1  \label{eq:prodx}.
\end{gather}
Geometrically, the quantity $\log X_{ij}$ is called the shear coordinate, which is the signed hyperbolic distance between the tangency points of the geodesic $ij$ with the respective incircles in the ideal triangles $ijk$ and $jkl$. Conversely, given a function  $X:E \to \mathbb{R}_{>0}$ satisfying \eqref{eq:prodx}, one obtains a complete hyperbolic metric with cusps by gluing ideal triangles with shear coordinates $\log X$.

A tangent vector to the Teichm\"{u}ller space can be described by the logarithmic derivative of the cross ratio, i.e. $x:= \frac{d}{dt} (\log X^{(t)}|_{t=0}$ where $X^{(t)}$ represents a path in $\Teich(S_{g,n})$. In such a way, every tangent vector corresponds to an element in a real vector space
\[
W^{\mathbb{R}}:=\{ x \in \mathbb{R}^{E}| \forall i \in V,   \sum_j x_{ij}=0\}.
\]
We have an identification of the tangent space
\begin{align}\label{eq:Tteich}
	T_X \Teich(S_{g,n}) \cong W^{\mathbb{R}} \cong  \mathbb{R}^{|E|-|V|}.
\end{align}

In order to introduce the symplectic form, we consider the decorated Teichm\"{u}ller space $\widetilde{\Teich}(S_{g,n})$. It consists of decorated hyperbolic structures, each of which represents a point in $\Teich(S_{g,n})$ together with a choice of horocycle $H_i$ for each puncture $i \in V$. By considering the hyperbolic length of the horocycles at the punctures, one has
\[
\widetilde{\Teich}(S_{g,n}) = \Teich(S_{g,n}) \times \mathbb{R}^{n}_{>0}.
\]
Any decorated hyperbolic structure yields a function  $A:E \to \mathbb{R}_{>0}$ such that $\log A_{ij}$ is the signed hyperbolic distance between the horocycles $H_i$ and $H_j$. It takes positive sign whenever the horocycles are disjoint. Such functions parametrize the decorated Teichm\"{u}ller space.

Similarly, a tangent vector to the decorated Teichm\"{u}ller space can be described by the logarithmic derivative, i.e. $a:= \frac{d}{dt} (\log A^{(t)})|_{t=0}$ where $A^{(t)}$ represents a path in $\Teich(S_{g,n})$. We have an identification of the tangent space
\[
T_A \widetilde{\Teich}(S_{g,n}) \cong \{ a:E \to \mathbb{R}\} = \mathbb{R}^{E}.
\]

By forgetting the horocycles, there is a natural projection
\begin{align*}
	\widetilde{\Teich}(S_{g,n}) &\to \Teich(S_{g,n}) \\
	A_{ij} &\mapsto X_{ij} = \frac{A_{ki} A_{lj}}{A_{il} A_{jk}} 
\end{align*}
and for the tangent space
\begin{align*}
	T_A \widetilde{\Teich}(S_{g,n}) &\to T_X \Teich(S_{g,n}) \\
	a_{ij} &\mapsto  x_{ij} = a_{ki}-a_{il}+a_{lj}-a_{jk}
\end{align*}
Particularly, the linear map 
\begin{align*}
h:\mathbb{R}^{E} &\to W^{\mathbb{R}} \\
         a& \mapsto x_{ij} = a_{ki}-a_{il}+a_{lj}-a_{jk}
\end{align*}
is surjective.

\begin{theorem}[Penner]
	The pullback of the Weil-Petersson symplectic 2-form $\omega_P$ on $\Teich(S_{g,n})$ is a bilinear form on $\widetilde{\Teich}(S_{g,n})$  as
	\begin{align*}
		\tilde{\omega}_{P} &:= -2 \sum_{ijk \in F} d \log A_{ij} \wedge d \log A_{jk} + d \log A_{jk} \wedge d \log A_{ki} + d \log A_{ki} \wedge d \log A_{ij}. 
	\end{align*}
	It is  invariant under change of horocycles and invariant under edge flipping in the triangulation.
\end{theorem}

For later calculations, we express the 2-form in another way.
\begin{corollary}\label{cor:symppair}
	Suppose $x,\tilde{x} \in T_X \Teich(S_{g,n})$ are two tangent vectors and we consider any of their lift $a,\tilde{a} \in T_A\widetilde{\Teich}(S_{g,n})$. Then
	\[
	\tilde{\omega}_{P}(a,\tilde{a}) =   2\sum_{ij \in E} a_{ij} \tilde{x}_{ij} = -2 \sum_{ij \in E} \tilde{a}_{ij} x_{ij}
	\]
\end{corollary}
\begin{proof}
	It follows from rewriting
	\begin{align*}
		\tilde{\omega}_{P}(a,\tilde{a}) &=-2 \sum_{ijk \in F} a_{ij} (\tilde{a}_{jk}-\tilde{a}_{ki}) + a_{jk} (\tilde{a}_{ki}-\tilde{a}_{ij}) + a_{ki} (\tilde{a}_{ij}-\tilde{a}_{jk}) 
	\end{align*}
with $x_{ij}=a_{ki}-a_{il}+a_{lj}-a_{jk}$.
\end{proof}

Motivated by the symplectic form on the Teichm\"{u}ller space $\Teich(S_{g,n})$, we equip the vector spaces  $W_{X}^{\mathbb{C}}$ and  $W_{X}^{\mathbb{R}}$ with an analogous skew-symmetric bi-linear form. We consider a complex vector space
\[
W^{\mathbb{C}}:=\{ x \in \mathbb{C}^{E} | \forall i \in V, \quad \sum_{j} x_{ij}=0 \}
\]
which is the complexification of $W^{\mathbb{R}}$ in Equation \eqref{eq:Tteich}. It includes $W_{X}^{\mathbb{C}}$ as a subspace.

\begin{definition}\label{def:wg}
	We define a skew-symmetric bi-linear forms $\widetilde{\omega}_{P}^{\mathbb{C}}$ on $\mathbb{C}^{E}$ as follows: for any $a, \tilde{a} \in \mathbb{C}^{E}$
	\[
	\widetilde{\omega}_{P}^{\mathbb{C}}(a,\tilde{a}):= -2 \sum_{ijk \in F} a_{ij} (\tilde{a}_{jk}-\tilde{a}_{ki}) + a_{jk} (\tilde{a}_{ki}-\tilde{a}_{ij}) + a_{ki} (\tilde{a}_{ij}-\tilde{a}_{jk}).
	\]
	Via the surjective mapping
	\begin{align*}
		h: \mathbb{C}^E &\to W^{\mathbb{C}}\\
		a_{ij} &\mapsto  x_{ij} = a_{ki}-a_{il}+a_{lj}-a_{jk}
	\end{align*}
	we define a skew-symmetric bilinear form 	$\omega_{P}^{\mathbb{C}}$ on $W^{\mathbb{C}}$ as follows: $\forall x,\tilde{x} \in W^{\mathbb{C}}$
	\[
	\omega_{P}^{\mathbb{C}}(x,\tilde{x}) := \widetilde{\omega}_{P}^{\mathbb{C}}(a,\tilde{a})
	\]
	where $a \in h^{-1}(x)$, $\tilde{a} \in h^{-1}(\tilde{x})$. The bilinear form $\omega_{P}^{\mathbb{C}}$ is well defined and independent of the choice of the preimages $a,\tilde{a}$. Moreover, $\omega_{P}^{\mathbb{C}}$ is non-degenerate on $W^{\mathbb{C}}$ and $(W^{\mathbb{C}},\omega_{P}^{\mathbb{C}})$ is a symplectic vector space. 
\end{definition}

Indeed, the non-degeneracy of $\omega_{P}^{\mathbb{C}}$ follows from that of the Weil-Petersson symplectic form $\omega_P$. We would like to explore how $\omega^{\mathbb{C}}_P$ behaves when restricted to the $W^{\mathbb{C}}_X$ and $W^{\mathbb{R}}_X$. Particularly, $\omega^{\mathbb{C}}_P|_{W^{\mathbb{R}}_X}$ is induced from the symplectic form on $\Teich(S_{g,n})$.

\begin{proposition}
	For any fixed Delaunay angle structure $\Theta$, the mapping from the space of circle patterns with prescribed intersection angles $\Theta$ to the Teichm\"{u}ller space of punctured surface 
	\begin{align*}
		\tilde{f}: P(\Theta) &\to \Teich(S_{g,n}) \\
		X &\mapsto |X|
	\end{align*}
	is injective. Furthermore, it induces a natural inclusion for every $X \in P(\Theta)$
	\begin{align*}
	i: W^{\mathbb{R}}_{X} &\xhookrightarrow{} T_{|X|}\Teich(S_{g,n})\\
		x &\mapsto x
\end{align*}
	with $i^*\omega_{P} = \omega^{\mathbb{C}}_P|_{W^{\mathbb{R}}_X}$.
\end{proposition}
\begin{proof}
	Suppose $\tilde{f}(X)=\tilde{f}(X')$ for some $X,X' \in P(\Theta)$. It implies $|X_{ij}|=|X'_{ij}|$ for all edges $ij$. On the other hand, $\arg X_{ij} = \Theta_{ij} = \arg X_{ij}'$. Thus $X_{ij}= X_{ij}'$ for all edges. Hence $\tilde{f}$ is an embedding.
\end{proof}

\begin{remark}
	Unlike the real cross ratios in the decorated Teichm\"{u}ller theory, generally given a complex cross ratio $X \in P(\Theta)$, there might be no $A:E \to \mathbb{C}$ such that 
	\begin{align}\label{eq:factorization}
		X_{ij} = \frac{A_{ki} A_{lj}}{A_{il} A_{jk}}.
	\end{align}
	An example is the one-vertex triangulation of the torus, where the factorization \eqref{eq:factorization} fails to exist. It is due to the fact that the complex cross ratios satisfy for every vertex $i$, 
	\[
	\sum_j \log X_{ij} = 2 \pi \sqrt{-1}
	\]
	where the right hand side is $2 \pi \sqrt{-1}$ instead of $0$.
\end{remark}

\section{Space of projective structures $P(S_g)$}\label{sec:proj}

 Following the survey \cite{Dumas2009}, we recall the space $P(S_g)$ of marked complex projective structures on a closed surface. We then introduce osculating M\"{o}bius transformations on circle patterns in order to pull back the symplectic form on $P(S_g)$ to $W^{\mathbb{C}}_X$.

\begin{definition}
	A complex projective structure on a closed surface $S_{g}$ is a maximal atlas of charts from open subsets of $S_{g}$ to the Riemann sphere such that the transition functions are restrictions of M\"{o}bius transformations. These charts are called projective charts.
	
	Two complex projective structures are marked isomorphic if there is a diffeomorphism homotopic to the identity mapping projective charts to projective charts. We denote $P(S_{g})$ the space of marked complex projective structures up to isomorphism.
\end{definition}
The topology of $P(S_g)$ is well-studied and
\[
P(S_g) \cong \mathbb{C}^{6g-6} \cong \mathbb{R}^{12g -12}.
\]
It includes Fuchsian structures (i.e. hyperbolic structures) and quasi-Fuchsian structures. Every complex projective structure $\sigma$ induces a developing map of the universal cover to the Riemann sphere whose holonomy yields a representation \[\rho \in \Hom( \pi_1(S_g), PSL(2,\mathbb{C})).\] 
Particularly, it satisfies for $\gamma_1,\gamma_2 \in \pi_1(S_g)$ 
\begin{equation}\label{eq:holrep}
	\rho_{\gamma_1 \gamma_2} = \rho_{\gamma_1} \rho_{\gamma_2}
\end{equation}
For a given complex projective structure, the developing map is unique up to a M\"{o}bius transformation while the holonomy representation is unique up to conjugation. Here two holonomy representations $\rho$ and $\tilde{\rho}$ are conjugate if there exists a constant $g\in PSL(2,\mathbb{C})$ such that for $\gamma \in \pi_1(S_g)$
\[
\tilde{\rho}_{\gamma} = g \rho_{\gamma} g^{-1}.
\]
We define the character variety 
\[
\mathcal{X}(S_g):= \Hom( \pi_1(S_g), PSL(2,\mathbb{C}))\sslash PSL(2,\mathbb{C}) 
\]
and the holonomy map
\[
\mbox{Hol}: P(S_g) \to \mathcal{X}(S_g).
\]
\begin{theorem}(Hejhal [47], Earle [28], Hubbard [51]). The holonomy map $\mbox{Hol} : P(S_g) \to \mathcal{X}(S_g)$ is a local biholomorphism.
\end{theorem}

By results of Goldman \cite{Goldman1984}, the tangent space of the character variety, and thus the tangent space of $P(S_g)$, are identified with the group cohomology. When one has a 1-parameter family of complex projective structures with holonomy $\rho^{(t)}$ satisfying $ \rho= \rho^{(t)}|_{t=0}$ and  $\dot{\rho}:= \frac{d}{dt}\rho^{(t)} |_{t=0}$, Equation \eqref{eq:holrep} implies that the mapping \[ \tau:= \dot{\rho} \rho^{-1} : \pi_1(S_g) \to sl(2,\mathbb{C})\]  satisfies a \textit{cocycle condition}: for every $\gamma_1,\gamma_2 \in \pi_1(S_g)$
\begin{align}\label{eq:cocycle}
	\tau_{\gamma_1 \gamma_2}=  \tau_{\gamma_1} + \Ad \rho_{\gamma_1} (\tau_{\gamma_2}).
\end{align}
where 
\[
\Ad \rho_{\gamma_1} (\tau_{\gamma_2})  :=  \rho_{\gamma_1} \tau_{\gamma_2} \rho^{-1}_{\gamma_1}.
\]
We write $Z^{1}_{\Ad \rho}(\pi_1(S_g),sl(2,\mathbb{C})$ the space of cocycles, which are functions $\tau: \pi_1(S_g) \to sl(2,\mathbb{C})$ satisfying the cocycle condition (Equation \eqref{eq:cocycle}). This space contains a subspace $B^{1}_{\Ad \rho}(\pi_1(S_g),sl(2,\mathbb{C}))$ of \textit{coboundaries}, which are functions  $\tau: \pi_1(S_g) \to sl(2,\mathbb{C})$ in the form
\[
\tau_\gamma = \tau_0- \Ad \rho_{\gamma} (\tau_0) 
\]
for some constant $\tau_0 \in sl(2,\mathbb{C})$. The coboundaries correspond to the trivial change of holonomy induced from conjugation. Indeed, suppose $g^{(t)} \in SL(2,\mathbb{C})$ is a path with $\mbox{Id}= g^{(t)}|_{t=0}$ and $\tau_0= \frac{d}{dt} g^{(t)}|_{t=0}$. By conjugation with $g^{(t)}$, we obtain a 1-parameter family of holonomy $\rho^{(t)}$ yielding for $\gamma \in \pi_1$
\[
\tau_{\gamma}= \left(\frac{d}{dt} \rho^{(t)}_{\gamma}|_{t=0}\right) \rho^{-1}_{\gamma} = \tau_0 - \Ad \rho_{\gamma} (\tau_0).
\] 
and hence $\tau:\pi_1(S_g) \to sl(2,\mathbb{C})$ is a coboundary. We then have an identification of the tangent space of $P(S_g)$ with the group cohomology defined as the quotient space
\[
T_{\sigma}P(S_g)\cong H^1_{\Ad\rho}(\pi_1(S_g),sl(2,\mathbb{C})) := \frac{Z^{1}_{\Ad\rho}(\pi_1(S_g),sl(2,\mathbb{C}))}{B^{1}_{\Ad \rho}(\pi_1(S_g),sl(2,\mathbb{C}))}
\]
 where $\rho$ is the holonomy map associated to the complex projective structure $\sigma$.
 
  We shall express the Goldman's symplectic form on $P(S_g)$ in terms of the wedge product in de Rham cohomology. We refer \cite{Labourie2013} for more detailed discussion.
 
 We denote  $\tilde{S}$ the universal cover of $S_g$ equipped with a complex projective structure by lifting $\sigma$. We construct a flat vector bundle $F_{\rho}$ over $(S_g,\sigma)$. 
 \[
 F_{\rho} := (\tilde{S} \times sl(2,\mathbb{C}))/ \pi_1(S_g)\] 
 where $\pi_1(S_g)$ acts on $ \tilde{S} \times  sl(2\,\mathbb{C})$ by
 \[
 \gamma \cdot (p,x) = (\gamma(p), \text{Ad}(\rho_{\gamma})(x)).
 \]
 for $\gamma \in \pi_1(S_g)$ where $\gamma (p)$ acts via the deck transformation. A lift of a $F_{\rho}$-valued 1-form on $S_g$ to the universal cover is a $sl(2,\mathbb{C})$-valued 1-form $\alpha$ satisfying
 \[
 \alpha \circ \gamma= \Ad \rho_{\gamma} ( \alpha)
 \]
 We write $Z^1(S_g,F_{\rho})$ the space of closed $F_{\rho}$-valued 1-forms. It contains the subspace $B^1_{\text{dR}}(S_g,F_{\rho})$ of exact 1-forms. We then denote $H^1_{\text{dR}}(S_g,F_{\rho})$ the de Rham cohomology on the surface $S_g$ with values in $F_{\rho}$ as the quotient space
 \[
 H^1_{\text{dR}}(S_g,F_{\rho}) := \frac{Z^1_{\text{dR}}(S_g,F_{\rho})}{B^1_{\text{dR}}(S_g,F_{\rho})}
 \]
 It is known that there is an isomorphism 
 \begin{align*}
 	\Phi:H^1_{\Ad \rho}(\pi_1(S_g),sl(2,\mathbb{C})) &\to H^1_{\text{dR}}(S_g,F_{\rho})
 \end{align*}
 given via the periods along loops. To illustrate this mapping, let $p\in \tilde{S}$ be a lift of the base point for the fundamental group. Let $\alpha$ be the lift of a closed $F_{\rho}$-valued 1-form. Then one defines $\tau:\pi_1(S_g) \to sl(2,\mathbb{C})$ such that for $\gamma \in \pi_1(S_g)$
 \[
 \tau_{\gamma}= \int_{p}^{\gamma (p)} \alpha.
 \]
 where the line integral on the universal cover is independent of path. For any $\gamma_1,\gamma_2 \in \pi_1(S_g)$, one has
 \[
 \tau_{\gamma_1 \gamma_2}=  \int_{p}^{\gamma_1 (p)} \alpha + \int_{p}^{\gamma_2 (p)} \alpha\circ \gamma_1=\tau_{\gamma_1 }+ \Ad\rho_{\gamma_1} (\tau_{\gamma_2})
 \]
 It yields that $\tau$ is a cocycle. If $\alpha$ is changed by a $F_{\rho}$-valued exact 1-form or another lift of the base point is used, then $\tau$ differs by a coboundary. Thus the isomorphism is defined over the quotient space such that $\Phi([\tau])=[\alpha]$. 
 
 On the other hand, since the universal cover is simply connected, one can consider a primitive function of the closed 1-form, namely $m:\tilde{S} \to sl(2,\mathbb{C})$ defined by
 \[
 m(x) := \int_p^x \alpha.
 \]
 Observe that $m(p)=0$ and
 \begin{align*}
 	\tau_{\gamma} =& m( \gamma(p)) - \Ad \rho_{\gamma} (m(p)) \\=& m( \gamma(x)) - \Ad \rho_{\gamma} (m(x)) - \int_{p}^{x} \left( \alpha\circ \gamma - \Ad \rho_{\gamma}(\alpha) \right) \\=&m( \gamma(x)) - \Ad \rho_{\gamma} (m(x))
 \end{align*}
 which remains constant for all $x \in \tilde{S}$.
 
 With the isomorphism $\Phi$, Goldman's symplectic form is defined such that for $\tau, \tilde{\tau} \in H^1_{\Ad \rho}(\pi_1(S_g),sl(2,\mathbb{C}))$
 \[
 \tilde{\omega}_G([\tau],[\tilde{\tau}])=  \iint_{S_g} \tr( \Phi(\tau) \wedge \Phi(\tilde{\tau}) ).
 \]
 We shall expand the integral to verify that indeed it is expressed in terms solely of $\tau$ and $\tilde{\tau}$. The expression will be needed for the proof of the main theorem. Consider a fundamental domain $\mathcal{F}$ on the universal cover by cutting $S_g$ along generators $\gamma_1,\gamma_2,\dots,\gamma_{2g} \in \pi_1(S_g)$ such that
 \[
 \gamma_1 \circ \gamma_2 \circ \gamma_1^{-1} \circ \gamma_2^{-1} \dots \gamma_{2g-1} \circ \gamma_{2g} \circ \gamma_{2g-1}^{-1} \circ \gamma_{2g}^{-1} =1 \in \pi_1(S_g).
 \]
 Then $\mathcal{F}$ is a 4g-polygon with sides matched in distinct pairs, i.e. the boundary is written as
 \[
 \partial \mathcal{F} = \tilde{\gamma}_1 + \tilde{\gamma}_2 + \tilde{\gamma}'_1 + \tilde{\gamma}'_2 + \dots \tilde{\gamma}_{2g-1} + \tilde{\gamma}_{2g} + \tilde{\gamma}'_{2g-1} + \tilde{\gamma}'_{2g}
 \]
 where $ \tilde{\gamma}_r$ and  $\tilde{\gamma}'_r$ are some lifts of the loops $\gamma_r$ and $\gamma_r^{-1}$ to the universal cover. For $r=1,2,\dots,2g$, there is a unique $\delta_r \in  \pi_1 (S_g)$ carrying a side $ \tilde{\gamma}_r$ to the paired side $ \tilde{\gamma}_r'$ reversing the orientation via a deck transformation. We write $\alpha$, $\tilde{\alpha}$ the lifts of the closed 1-forms representing $\Phi(\tau)$ and $\Phi(\tilde{\tau})$. Let $m:\tilde{S} \to sl(2,\mathbb{C})$ be a primitive function of $\alpha$. With these, we expand the integral using Stokes' theorem
 \begin{align}\label{eq:wgexplicit}
 	\begin{split}
 		\tilde{\omega}_G(\tau, \tilde{\tau})  &=  \int_{\partial F_g}  \tr (m \tilde{\alpha}) \\
 		&= \sum_{r=1}^{2g}\left( \int_{\gamma_r} \tr( m\tilde{\alpha}) -  \int_{\gamma_r} \tr( m\circ \delta_r \cdot \tilde{\alpha}\circ \delta_r) \right) \\
 		&= \sum_{r=1}^{2g}\left( \int_{\gamma_r} \tr( \Ad \rho_{\delta_r} (m) \cdot \Ad \rho_{\delta_r} (\tilde{\alpha})) -  \int_{\gamma_r} \tr( m\circ \delta_r \cdot \tilde{\alpha}\circ \delta_r) \right) \\ 
 		&= \sum_{r=1}^{2g}\left( \int_{\gamma_r} \tr( \Ad \rho_{\delta_r} (m) \cdot \tilde{\alpha} \circ \delta_r) -  \int_{\gamma_r} \tr( m\circ \delta_r \cdot \tilde{\alpha}\circ \delta_r) \right) \\
 		&= -\sum_{r=1}^{2g}\tr\left( (m\circ \delta_r - \Ad \rho_{\delta_r} (m))    (\int_{\gamma_r} \tilde{\alpha} \circ \delta_r) \right) \\
 		&= -\sum_{r=1}^{2g}\tr\left( \tau_{\delta_r}  (\int_{\gamma_r} \tilde{\alpha} \circ \delta_r) \right)
 	\end{split}
 \end{align}
 where we used the fact that $\left(m\circ \delta_r - \Ad \rho_{\delta_r} (m)\right)$ is a constant function on the universal cover and equal to $ \tau_{\delta_r}$. On the other hand, the line integral 
 \[
 \int_{\gamma_r} \tilde{\alpha} \circ \delta_r
 \]
 is the evaluation of $\tilde{\tau}$ at some element in $\pi_1(S_g)$. It is known that both $\Re(\tilde{\omega}_G)$ and $\Im(\tilde{\omega}_G)$ define real symplectic forms on $P(S_g)$ and in particular, non-degenerate.

\subsection{Change of holonomy via osculating M\"{o}bius transformations}

Every cross ratio $X\in P(\Theta)$ induces a developing map of the triangulation of the universal cover. Similarly, every element in $ W^{\mathbb{C}}_X$ induces an infinitesimal deformation of the developing map and thus a change in the holonomy. We explain the connection by introducing osculating M\"{o}bius transformation for a pair of circle patterns. Osculation M\"{o}bius transformations was used in \cite{Lam2024} to construct discrete constant-mean-curvature-1 surfaces in hyperbolic 3-space.

\begin{proposition}
	Given a cross ratio system $X:E \to \mathbb{C}$. We denote $(\hat{V},\hat{E},\hat{F})$ the pullback triangulation on the universal cover and lift $X$ to the universal cover such that it is invariant under deck transformations.  Then there exists a \emph{developing map} of the triangulation $z:\hat{V} \to \mathbb{C}\cup \{ \infty \}$ such that for every edge $ij \in \hat{E}$ shared by triangles $\{ijk\}$ and $\{jil\}$
	\begin{equation}\label{eq:developingmap}
		X_{ij} = -\frac{(z_k - z_i)(z_l -z_j)}{(z_i - z_l)(z_j - z_k)}.
	\end{equation}
	The developing map is unique up to complex projective transformations and defines the same holonomy representation from the underlying complex projective structure.
\end{proposition}

Given two Delaunay circle patterns $X,\tilde{X}:E \to \mathbb{C}$ with corresponding developing maps $z,\tilde{z}:\hat{V} \to \mathbb{C}P^1$ respectively, it yields a function $M:\hat{F} \to SL(2,\mathbb{C})/\{\pm I\}$ where for every face $ijk\in \hat{F}$, $M_{ijk}$ corresponds to a M\"{o}bius transformation mapping vertices $z_i,z_j,z_k$ to $\tilde{z}_i,\tilde{z}_j,\tilde{z}_k$. We call $M$ the osculating M\"{o}bius transformation. For any edge $ij$ shared by two faces $ijk$ and $jil$, one could deduce that $z_i$ and $z_j$ are the fixed points of the M\"{o}bius transformation $ M_{jil}^{-1}M_{ijk}$. Thus $M_{jil}^{-1}M_{ijk}$ have eigenvectors $\left(\begin{array}{c} z_i \\ 1
\end{array} \right)$ and $\left(\begin{array}{c} z_j \\ 1
\end{array} \right)$ (See \cite[Proposition 2.2]{Lam2024} for the explicit formula). On the other hand, for any face $ijk \in \hat{F}$ and $\gamma \in \pi_1(S_g)$, the M\"{o}bius transformations corresponding to $M_{\gamma(ijk)} \circ \rho_{\gamma}$ and  $\tilde{\rho}_{\gamma} \circ M_{ijk}$ respectively map $z_i,z_j,z_k$ to $\tilde{z}_{\gamma(i)},\tilde{z}_{\gamma(j)},\tilde{z}_{\gamma(k)}$. Thus we deduce that 
\[
M_{\gamma(ijk)} \rho_{\gamma} = \tilde{\rho}_{\gamma} M_{ijk}
\]
which is equivalent to 
\begin{align}\label{eq:chmon}
	\tilde{\rho}_{\gamma} \rho_{\gamma}^{-1}=M_{\gamma(ijk)} \rho_{\gamma} M^{-1}_{ijk} \rho_{\gamma}^{-1}.
\end{align}

Suppose $X^{(t)}:E \to \mathbb{C}$ is a 1-parameter family of cross ratios satisfying Equation \eqref{eq:crproduct} \eqref{eq:crsum} with $X=X^{(t)}|_{t=0}$. We denote $x:= \frac{d}{dt}(\log X^{(t)})|_{t=0}$, $\dot{\rho}:= \frac{d}{dt}\rho^{(t)}|_{t=0}$ and $m:= \frac{d}{dt}M^{(t)}|_{t=0}$. Particularly we have $m:\hat{F} \to sl(2,\mathbb{C})$. Equation \eqref{eq:chmon} implies that the function $\tau:\pi_1(S_g) \to sl(2,\mathbb{C})$ defined such that for any $\gamma \in \pi_1$
\[
\tau_{\gamma}:= \dot{\rho}_{\gamma}\rho^{-1}_{\gamma}= m_{\gamma(ijk)} - \Ad \rho_{\gamma} (m_{ijk})
\]
is independent of $ijk \in \hat{F}$. Hence $\tau$ is a well-defined cocycle representing an element in $H^1_{\Ad \rho}(\pi_1(S_g),sl(2,\mathbb{C}))$. On the other hand, for every oriented edge $ij$ with left face $ijk$ and right face $jil$, the matrix $m_{ijk}-m_{jil}$ corresponds to an infinitesimal M\"{o}bius transformation that vanish at $z_i$ and $z_j$. It indicates that the matrix has eigenvector $\left(\begin{array}{c} z_i \\ 1
\end{array} \right)$ and $\left(\begin{array}{c} z_j \\ 1
\end{array} \right)$. One can show that the eigenvalue is $x_{ij}$ and $-x_{ij}$ (See \cite[Corollary 5.3]{Lam2015a}). Explicitly, we have
\[
m_{ijk}-m_{jil} = \frac{x_{ij}}{z_j-z_i} \left(
\begin{array}{cc}
	\frac{z_i+z_j}{2 } & -z_i
	z_j \\
	1 & \frac{-z_i-z_j}{2} \\
\end{array}
\right).
\]
where we lift $x$ to the universal cover such that it is invariant under deck transformations. In the following, we reverse the construction. Given an element $x \in W^{\mathbb{C}}_X$, we shall produce the osculation M\"{o}bius transformation $m$ and then a cocycle $\tau$.

\begin{proposition}\label{prop:equialpha}
	Suppose $X\in P(\Theta)$ with developing map $z$ and holonomy representation $\rho$. Given $x \in W^{\mathbb{C}}_X$, we lift it to the universal cover such that 
	it is invariant under deck transformations. Then it induces a function $\alpha: \vec{E} \to sl(2,\mathbb{C})$ on the oriented edges $\vec{E}$ on the universal cover
	\begin{equation}
			\alpha(\vec{ij}):= \frac{x_{ij}}{z_j-z_i} \left(
			\begin{array}{cc}
					\frac{z_i+z_j}{2 } & -z_i
					z_j \\
					1 & \frac{-z_i-z_j}{2} \\
				\end{array}
			\right) = -\alpha(\vec{ji}).
		\end{equation}
	It is a closed 1-form on the dual graph, in the sense that for every $i \in \hat{V}$
	\[
	\sum_{j} \alpha(\vec{ij}) =0.
	\]
	Furthermore, it is equivariant with respect to the fundamental group, namely for every $\gamma \in \pi_1(M)$
	\[
	\alpha\circ \gamma = \rho_{\gamma} \alpha  \rho^{-1}_{\gamma} = \Ad \rho_{\gamma} (\alpha)
	\]
\end{proposition}
\begin{proof}
		By substituting \eqref{eq:developingmap}, Equation \eqref{eq:xcrsum} is equivalent to the sum around vertex $i$
	\begin{align*}
		\sum_j \frac{x_{ij}}{z_j - z_i} &=0
	\end{align*} 
	Together with Equation \eqref{eq:xcrproduct}, we deduce that for every vertex $i \in \hat{V}$
	\begin{align*}
		\sum_j \frac{x_{ij}}{z_j - z_i} \frac{z_i+z_j}{2} = \frac{1}{2} \sum_j x_{ij} +  z_i \sum_j \frac{x_{ij}}{z_j - z_i} =0
	\end{align*}
   while
		\begin{align*}
		\sum_j \frac{x_{ij}}{z_j - z_i} z_i z_j = z_i\sum_j x_{ij} +  z^2_i \sum_j \frac{x_{ij}}{z_j - z_i} =0
	\end{align*}
	Thus, the closeness of $\alpha$ follows. 
	
	Let $\left( 	\begin{array}{cc}
		a & b
	   \\ c
		 & d \\
	\end{array}  \right) \in SL(2,\mathbb{C})$ and write $\tilde{z}= \frac{a z +b}{c z +d}$ be the image of $z$ under the associated M\"{o}bius transformation. Then with direct computation, we always have
\[
\left(
\begin{array}{cc}
	\frac{\tilde{z}_i+\tilde{z}_j}{2 (\tilde{z}_j-\tilde{z}_i)} & \frac{-\tilde{z}_i
	\tilde{z}_j}{\tilde{z}_j-\tilde{z}_i} \\
	\frac{1}{\tilde{z}_j-\tilde{z}_i} & \frac{-\tilde{z}_i-\tilde{z}_j}{2(\tilde{z}_j-\tilde{z}_i)} \\
\end{array}
\right) =  \left( 	\begin{array}{cc}
	a & b
	\\ c
	& d \\
\end{array}  \right) \left(
\begin{array}{cc}
	\frac{z_i+z_j}{2(z_j-z_i) } & \frac{-z_i
	z_j}{z_j-z_i} \\
	\frac{1}{z_j-z_i} & \frac{-z_i-z_j}{2(z_j-z_i)} \\
\end{array}
\right) \left( 	\begin{array}{cc}
	a & b
	\\ c
	& d \\
\end{array}  \right)^{-1}.
\]
Since for any $\gamma \in \pi_1(S_g)$ the vertices $z_{\gamma(i)}$ and $z_{\gamma(j)}$ are the image of $z_i$ and $z_j$ under the M\"{o}bius transformation of $\rho_{\gamma}$, thus it implies that for every oriented edges $\vec{ij}$
\[
\alpha(\gamma(\vec{ij}))= \rho_{\gamma} \alpha(\vec{ij}) (\rho_{\gamma})^{-1}
\]
and hence the equivariance holds.
\end{proof}

\begin{proposition}
	Suppose $X\in P(\Theta)$ with developing map $z:\hat{V} \to \mathbb{C}$ and holonomy representation $\rho$. Then for every $x \in W^{\mathbb{C}}_X$, there is a function $m:\hat{F} \to sl(2,\mathbb{C})$ unique up to a constant such that for every oriented edge $ij$,
\[
m_{ijk}-m_{jil} = \frac{x_{ij}}{z_j-z_i} \left(
\begin{array}{cc}
	\frac{z_i+z_j}{2 } & -z_i
	z_j \\
	1 & \frac{-z_i-z_j}{2} \\
\end{array}
\right).
\]
There is a well-defined complex linear map 
\begin{align*}
		\hol: W_X^{\mathbb{C}} &\to H^1_{\Ad\rho}(\pi_1(S_g),sl(2,\mathbb{C})) \\
		x &\mapsto [\tau]
\end{align*}
where $\tau:\pi_1(S_g) \to sl(2,\mathbb{C})$ is a cocycle satisfying for any $\gamma \in \pi_1$ we have
\[
\tau_{\gamma}= m_{\gamma(ijk)} - \Ad \rho_{\gamma} (m_{ijk})
\]
independent of $ijk \in \hat{F}$.
\end{proposition}
\begin{proof}
From the previous Proposition \ref{prop:equialpha}, $\alpha$ is a closed 1-form on the dual graph. Since the universal cover is simply connected, it can be integrated to obtain $m:\hat{F} \to sl(2,\mathbb{C})$ such that
\[
m_{ijk}-m_{jil} = \alpha(\vec{ij}).
\]

For any fixed $ijk \in \hat{F}$, we define a function $\tau:\pi_1(S_g) \to sl(2,\mathbb{C})$ such that for any $\gamma \in \pi_1$
\[
\tau_{\gamma}= m_{\gamma(ijk)} - \Ad \rho_{\gamma} (m_{ijk})
\]
It is independent of the face chosen because of the equivariance of $\alpha$. Indeed, if $jil \in \hat{F}$ is adjacent to $ijk$, then
\begin{align*}
m_{\gamma(jil)} - \Ad \rho_{\gamma} (m_{jil})=& m_{\gamma(ijk)} - \Ad \rho_{\gamma} (m_{ijk}) + \alpha(\gamma(\vec{ij}))- \Ad \rho_{\gamma}  (\alpha(\vec{ij}))  \\
=&  m_{\gamma(ijk)} - \Ad \rho_{\gamma} (m_{ijk}) \\=& \tau_{\gamma}.
\end{align*}
The function $\tau$ is a cocycle since for any $\gamma_1, \gamma_2 \in \pi_1(S_g)$
\begin{align*}
\tau_{\gamma_1 \gamma_2} =& m_{\gamma_1 \gamma_2(ijk)} - \rho_{\gamma_1} \rho_{\gamma_2} m_{ijk} \rho_{\gamma_2}^{-1} \rho_{\gamma_1}^{-1} \\
=& m_{\gamma_1(\gamma_2(ijk))} - \rho_{\gamma_1} m_{\gamma_2(ijk)} \rho_{\gamma_1}^{-1}+ \rho_{\gamma_1}(m_{\gamma_2(ijk)}- \rho_{\gamma_2} m_{ijk} \rho_{\gamma_2}^{-1}) \rho_{\gamma_1}^{-1} \\=&  \tau_{\gamma_1} + \Ad \rho_{\gamma_1}( \tau_{\gamma_2})
\end{align*}
If the function $m$ is added by a constant, then $\tau$ differs by a coboundary. Hence the mapping $[\tau]$ is well defined.
\end{proof}

\begin{definition}
	With the linear map $\hol$, we equip $W^{\mathbb{C}}_X$ with a skew symmetric bilinear form $\omega_G$ defined via
	\[
	\omega_G(x,\tilde{x}):= \tilde{\omega}_G(\hol(x), \hol(\tilde{x})) 
	\]
	which is the pullback of the symplectic form on $P(S_g)$. 
\end{definition}

\section{Pullbacks of symplectic forms coincide on $P(\Theta)$}

Observe that the function $\alpha$ in Proposition \ref{prop:equialpha} is a 1-cocycle in cellular cohomology, with the cell decomposition $(V^*,E^*,F^*)$ dual to our triangulation $(\hat{V},\hat{E},\hat{E})$ of the universal cover. We have bijections $\hat{V} \cong F^*,\hat{E} \cong E^*, \hat{F} \cong F^*$. Particularly,  $(V^*,E^*,F^*)$ is trivalent (See Figure \ref{fig:dualface}).

We shall compute the cup product in cellular cohomology and express it in two ways. We than show that they coincide with $\omega_G$ and $\omega_{P}$ respectively. Somehow our computation for $\omega_{P}$ is in line with Papadopoulos and Penner's derivation of $\omega_P$ from Thurston's symplectic form on measured laminations \cite{PP1993}.

Recall that in cellular cohomology, the cup-product of two 1-cocycles $\alpha, \tilde{\alpha}$ is a 2-cocycle $\alpha \cup \tilde{\alpha}$, which can be evaluated on faces. Suppose $\phi \in F^*$ is a triangular face on the universal cover with oriented edges $e_1, e_2, e_3 \in \partial \phi$. Then 
\begin{align}\label{eq:tricup}
	\begin{split}
			\alpha \cup \tilde{\alpha}(\phi):=&\frac{1}{6}\left(\alpha(e_1)\tilde{\alpha}(e_2)+\alpha(e_2)\tilde{\alpha}(e_3)+\alpha(e_3)\tilde{\alpha}(e_1)\right)\\&-\frac{1}{6}\left(\alpha(e_1)\tilde{\alpha}(e_3)+\alpha(e_2)\tilde{\alpha}(e_1)+\alpha(e_3)\tilde{\alpha}(e_2)\right)
	\end{split}
\end{align}
takes values in $gl(2,\mathbb{C})$. Generally if $\phi$ is a $k+2$-polygon with $k\geq1$, we triangulate $\phi$ into triangular faces $\phi_1\dots \phi_k$ by adding diagonals. There are unique extensions of $\alpha$ and $\tilde{\alpha}$ to the diagonals such that their summation over $\partial \phi_r$ vanish for all $r=1,\dots k$. With these extensions and Equation \eqref{eq:tricup}, we define
\[
\alpha \cup \tilde{\alpha}(\phi):=\sum_{r=1}^k 	\alpha \cup \tilde{\alpha}(\phi_r)
\]
which is independent of the way how to triangulate $\phi$. An elementary way to see this independence is to verify the invariance under edge flipping. Because $\alpha,\tilde{\alpha}$ are equivariant with respect to $\rho$, we have for any face $\phi$ and for $\gamma \in \pi_1(S_{g})$
\[
\tr(\alpha \cup \tilde{\alpha}(\gamma\cdot\phi)) = \tr((\rho_{\gamma} \alpha  \rho^{-1}_{\gamma}) \cup (\rho_{\gamma} \tilde{\alpha}  \rho^{-1}_{\gamma})(\phi)) = \tr(\alpha \cup \tilde{\alpha} (\phi)) 
\]
is invariant under the deck transformation.

We take a fundamental domain $\mathcal{F}$ on the universal cover as in Section \ref{sec:proj} and define
\begin{equation}\label{eq:simcup}
	\omega(x,\tilde{x}):= \sum_{\phi \in \mathcal{F}} \tr(\alpha \cup \tilde{\alpha}(\phi)) \in \mathbb{C}
\end{equation}
where the sum is over all the faces in $\mathcal{F}$. The quantity is independent of the choice of the fundamental domain.

We evaluate $ \tr(\alpha \cup \tilde{\alpha}(\phi))$ for every face with the following observation.

\begin{lemma}\label{eq:traceprod} For any $z_i,z_j,z_k,z_l \in \mathbb{C}$ such that $z_i \neq z_j$ and $z_k \neq z_l$,  we have
	\begin{align*}
		\tr(\left(
		\begin{array}{cc}
			\frac{z_i+z_j}{2 (z_j-z_i)} & -\frac{z_i
				z_j}{z_j-z_i} \\
			\frac{1}{z_j-z_i} & \frac{-z_i-z_j}{2 (z_j-z_i)} \\
		\end{array}
		\right).\left(
		\begin{array}{cc}
			\frac{z_k+z_l}{2 (z_l-z_k)} & -\frac{z_k
				z_l}{z_l-z_k} \\
			\frac{1}{z_l-z_k} & \frac{-z_k-z_l}{2 (z_l-z_k)} \\
		\end{array}
		\right) ) &=  \frac{1}{2}-\frac{(z_i-z_k) (z_j-z_l)}{(z_i-z_j)
			(z_k-z_l)}.
	\end{align*}
	In particular, it is equal to $1/2$ if  $z_i =z_k$ or $z_j=z_l$. It is equal to $-1/2$ if $z_i =z_l$ or $z_j=z_k$.
\end{lemma}

\begin{proposition}\label{prop:pencup}
 Let $X \in P(\Theta)$ represent a Delaunay circle pattern on a surface with complex projective structure. Then for any $x,\tilde{x} \in W_X^{\mathbb{C}}$, we have
\[
\omega(x,\tilde{x}) = \frac{1}{2} \omega_{P}^{\mathbb{C}}(x,\tilde{x}) 
\]	
\end{proposition}
\begin{proof}
	Given $x,\tilde{x} \in W_X^{\mathbb{C}}$, via Proposition \ref{prop:equialpha}, we obtain two functions $\alpha,\tilde{\alpha}:\vec{E} \to sl(2,\mathbb{C})$. We shall compute Equation \eqref{eq:simcup}.

	Suppose $\phi_i \in F^*$ is a face dual to vertex $i \in \hat{V}$. Observe that for all edges within $\phi_i$, both $\alpha$ and $\tilde{\alpha}$ take values in $sl(2,\mathbb{C})$-matrices sharing a common eigenvector $\left(\begin{array}{c} z_i \\ 1
	\end{array} \right)$. Then Lemma \ref{eq:traceprod} applies and implies that the term $\tr(\alpha \cup \tilde{\alpha}(\phi_i))$ involves only $x$ and $\tilde{x}$ but not the developing map $z$ at all. We further expand $x$ into some $a \in h^{-1}(x)$ and focus on the coefficients of $a_{ij}$ for an edge $ij$ in the summation $\eqref{eq:simcup}$. We shall argue that the coefficient of $a_{ij}$ in the summation is $\tilde{x}_{ij}$.
	
	By the definition of the cup product, there are only four faces in $F^*$ that contribute terms involving $a_{ij}$ (See Figure \ref{fig:dualface}). 
	\begin{figure}
		\centering
					\includegraphics[width=0.8\textwidth]{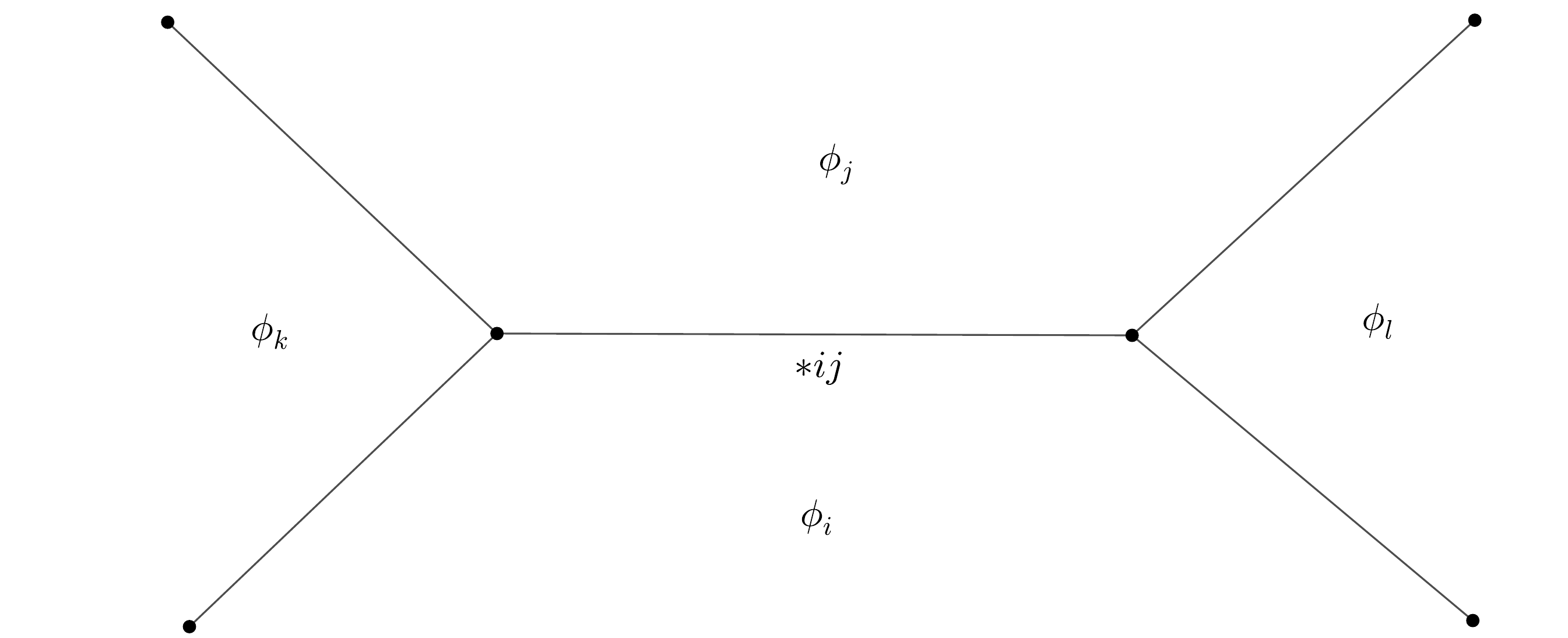}
			\caption{The edge $*ij$ is the dual edge of $ij$. It is the common edge of the dual faces $\phi_i$ and $\phi_j$.}
			\label{fig:dualface}
	\end{figure}
	Two of the faces $\phi_i$ and $\phi_j$ share a common edge $*ij$ which is the dual edge of $ij$. Two faces $\phi_j$ and $\phi_k$ share a corner with $*ij$. Analyzing the cup product, one finds that the coefficients of $a_{ij}$ in $\tr(\alpha \cup \tilde{\alpha}(\phi_i))$ and $\tr(\alpha \cup \tilde{\alpha}(\phi_j))$ are respectively
	\[
	\frac{1}{12}(6 \tilde{x}_{ij} +3 \tilde{x}_{ik} +3\tilde{x}_{il}) ) \quad \text{and} \quad 	\frac{1}{12}(6 \tilde{x}_{ij} +3 \tilde{x}_{jk} +3\tilde{x}_{jl}) )
	\]
	On the other hand, the coefficients of $a_{ij}$ in $\tr(\alpha \cup \tilde{\alpha}(\phi_k))$ and $\tr(\alpha \cup \tilde{\alpha}(\phi_l))$ are respectively
	\[
		\frac{1}{12}(-3 \tilde{x}_{ik} -3\tilde{x}_{jk}) ) \quad \text{and} \quad 	\frac{1}{12}(-3 \tilde{x}_{il} -3\tilde{x}_{jl}) ).
	\]
	Thus the coefficient of $a_{ij}$ in the total sum \eqref{eq:simcup} is $ \tilde{x}_{ij}$, which is the same as in $ \frac{1}{2} \omega_{P}^{\mathbb{C}}(x,\tilde{x}) $ by Corollary \ref{cor:symppair}. Since it holds for arbitrary edges $ij$, we obtain the claim.
\end{proof}

We now proceed to evaluate $\omega(x,\tilde{x})$ in terms of the periods over closed loops and show that it coincides $\omega_{G}(x,\tilde{x})$, which is defined using the wedge product in de Rham cohomology. Actually it should follow from a general theory of cohomology. Here we take an elementary way to verify it.

\begin{proposition}\label{prop:goldcup}
	Let $X \in P(\Theta)$ represent a Delaunay circle pattern on a surface with complex projective structure. Then for any $x,\tilde{x} \in W_X^{\mathbb{C}}$, we have 
	\[
	\omega(x,\tilde{x}) = \omega_{G}(x,\tilde{x}) 
	\]	
\end{proposition}
\begin{proof}
	Given $x,\tilde{x} \in W_X^{\mathbb{C}}$, we produce two $sl(2,\mathbb{C})$ 1-forms $\alpha$, $\tilde{\alpha}$ on the universal cover that are closed on the dual faces. We write $m:V^{*} \to sl(2,\mathbb{C})$ the primitive of $\alpha$. Here $V^*$ denotes the set of dual vertices, each of which correspond to a face in $\hat{F}$. If $e$ is an oriented edge, then we write
	\[
	\alpha(e) = m_{t(e)} - m_{s(e)}.
	\]
	where $t(e), s(e) \in V^*$ are the tail and the start of $e$. In other words, $e$ is an edge oriented from $s(e)$ to $t(e)$.
	
	When computing the cup product, we triangulated the cellular decomposition $(V^*,E^*,F^*)$ by adding diagonals. Suppose $\phi$ is one of such triangular faces with oriented edges $e_1, e_2, e_3 \in \partial \phi$. We write $v_1,v_2,v_3$ the vertices opposite to $e_1, e_2, e_3$. Then 
	\begin{align*}
		\alpha \cup \tilde{\alpha}(\phi)=&\frac{1}{6}((m_{v_2}+m_{v_3}-2m_{v_1})\tilde{\alpha}(e_1) + (m_{v_3}+m_{v_1}-2m_{v_2}) \tilde{\alpha}(e_2)\\&+(m_{v_1}+m_{v_2}-2m_{v_3}) \tilde{\alpha}(e_3)) \\
		   =& -\frac{1}{2} (m_{v_1} \tilde{\alpha}(e_1)+m_{v_2} \tilde{\alpha}(e_2)+m_{v_3} \tilde{\alpha}(e_3))
	\end{align*}
  since $\tilde{\alpha}(e_1)+\tilde{\alpha}(e_3)+\tilde{\alpha}(e_3)=0$. If $v$ is an interior vertex of the fundamental domain $\mathcal{F}$, then the coefficient of $m_v$ in \begin{align}
   	\sum_{\phi \in \mathcal{F}} \alpha \cup \tilde{\alpha}(\phi) 
   \end{align} is zero. If $v$ is a boundary vertex, then the coefficient of $m_v$ is $\frac{1}{2}(\tilde{\alpha}(e_1)+\tilde{\alpha}(e_2))$ where $e_1,e_2 \in \partial \mathcal{F}$ are oriented edges with $v=t(e_1)$ the tail of $e_1$ and $v=s(e_2)$ the start of $e_2$. In summary,
   \[
   \sum_{\phi \in \mathcal{F}} \alpha \cup \tilde{\alpha}(\phi) = \sum_{e \in \partial \mathcal{F}}  \frac{m_{s(e)}+m_{t(e)}}{2}\tilde{\alpha}(e).
   \]
	Hence
	\begin{align*}
		&\omega(x,\tilde{x})\\
		=& \sum_{\phi \in \mathcal{F}} \tr(\alpha \cup \tilde{\alpha}(\phi)) \\
		=&\sum_{r=1}^{2g}\tr\left( \sum_{e\in \tilde{\gamma}_r} \frac{m_{s(e)}+m_{t(e)}}{2}\tilde{\alpha}(e)-\sum_{e\in \tilde{\gamma}_r} \frac{(m\circ \delta_r)_{s(e)}+(m\circ \delta_r) _{t(e)}}{2}\tilde{\alpha}\circ \delta_r (e) \right)\\
		=&\sum_{r=1}^{2g}\tr\left( \sum_{e\in \tilde{\gamma}_r} Ad_{\delta_r}(\frac{m_{s(e)}+m_{t(e)}}{2})\tilde{\alpha}\circ \delta_r (e)-\sum_{e\in \tilde{\gamma}_r} \frac{(m\circ \delta_r)_{s(e)}+(m\circ \delta_r) _{t(e)}}{2}\tilde{\alpha}\circ \delta_r (e) \right)\\
		=&-\sum_{r=1}^{2g}\tr\left( (m\circ \delta_r - \Ad_{\delta_r} m) \sum_{e\in \tilde{\gamma}_r} \tilde{\alpha}\circ \delta_r (e)    \right) \\
		=&-\sum_{r=1}^{2g}\tr\left( \tau_{\delta_r} \sum_{e\in \tilde{\gamma}_r} \tilde{\alpha}\circ \delta_r (e)    \right) \\
		=&\omega_{G}(x,\tilde{x}) 
	\end{align*}
where we used the fact that $(m\circ \delta_r - \Ad_{\delta_r} m)$ is a constant function and equal to $ \tau_{\delta_r} $. The last equality follows by comparing with Equation \eqref{eq:wgexplicit} and from the observation that the sum
\[
\sum_{e\in \tilde{\gamma}_r} \tilde{\alpha}\circ \delta_r (e) 
\]
is the evaluation of $\tilde{\tau}$ at some element in $\pi_1(S_g)$.
\end{proof}

\begin{proof}[Proof of Theorem \ref{thm:main}]
	It follows from Proposition \ref{prop:pencup} and Proposition \ref{prop:goldcup} that over $W_X^{\mathbb{C}}$
	\[
	\omega_{G} =	\omega =\frac{1}{2} \omega_{P}^{\mathbb{C}}
	\]

	Restricting to $W_X^{\mathbb{R}}$, we deduce that $2\omega_{G}=\omega^{\mathbb{C}}_{P}$ is expressed solely in real numbers and hence takes real values. It implies that the imaginary part $\Im \tilde{\omega}_{G}$ vanishes on the image $\hol( W_X^{\mathbb{R}})  $. Since $\Im \tilde{\omega}_G$ is a real symplectic form, particularly non-degenerate over $H^1_{\Ad\rho}(\pi_1(S_g),sl(2,\mathbb{C}))$, a classical result from symplectic geometry states that any isotropic subspace has dimension at most half of the dimension of the total space \cite{Cannas2001}. Hence, the image $\hol( W_X^{\mathbb{R}}) $ has real dimension at most $6g-6$.
\end{proof}

\begin{proof}[Proof of Corollary \ref{cor:infimplysmooth}]
	If the linear map $\hol$ has trivial kernel, Theorem \ref{thm:main} implies $\dim_{\mathbb{R}}W_X^{\mathbb{R}} \leq 6g-6$. Together with the lower bound (Equation \ref{eq:VRdim}), we deduce that  $\dim_{\mathbb{R}}W_X^{\mathbb{R}}= 6g-6$. By the constant rank theorem, we conclude that in a neighborhood of $X$, $P(\Theta)$ is a real analytic manifold of dimension $6g-6$ and the claims follow .  
\end{proof}

\section{Symplectic complement of $W_{X}^{\mathbb{C}}$ in $W^{\mathbb{C}}$}\label{sec:symcom}

We collect some intriguing observations about the symplectic complement, which is related to the smoothness of the deformation space and the non-degeneracy of the symplectic form.

To simplify the notations, we assume that no edge form a loop connecting to the same vertex on the surface. Recall that from Definition \ref{def:wg} that we have a symplectic vector space $(W^{\mathbb{C}}, \omega^{\mathbb{C}}_{P})$. It contains $W_{X}^{\mathbb{C}}$ as a subspace for any Delaunay circle pattern $X \in P(\Theta)$. We shall explore its symplectic complement $(W_{X}^{\mathbb{C}})^{\omega}$.

For every vertex $i\in V$ with adjacent vertices numbered as $1$, $2$, ..., $r$, we define $a^{(i)}:E \to \mathbb{C}$ such that for $j=1,2,\dots,r$,
\begin{align*}
	a^{(i)}_{i1}&=
	X_{i1} + X_{i1} X_{i2} + X_{i1} X_{i2} X_{i3} +\dots + X_{i1}X_{i2}\dots X_{ir} = 0 \\
	a^{(i)}_{i2}&=
	 X_{i1} X_{i2} +  X_{i1} X_{i2} X_{i3} + \dots + X_{i1}X_{i2}\dots X_{ir} = -(X_{i1}) \\
	 	a^{(i)}_{i3}&=
	 X_{i1} X_{i2} X_{i3}+ \dots + X_{i1}X_{i2}\dots X_{ir} = -(X_{i1}+X_{i1}X_{i2}) \\
	 &\;\;\vdots \notag \\
	a^{(i)}_{ij}&=
-(X_{i1} + X_{i1} X_{i2} + \dots + X_{i1}X_{i2}\dots X_{i\,j-1})
\end{align*}
and $a^{(i)}_{kl}=0$ for edges not connecting to $i$. Recall the surjective linear map $h: \mathbb{C}^{E} \to W^{\mathbb{C}}$, where $h(a)_{ij} = a_{ki}-a_{il}+a_{lj}-a_{jk}$.

\begin{proposition}
	For each $i\in V$, the element $x^{(i)}:=h(a^{(i)}) \in W_{X}^{\mathbb{C}}$ corresponds to the change in logarithmic cross ratios under an infinitesimal deformation of a single vertex $i$ while other vertices and the complex projective structure keep fixed.
\end{proposition}
\begin{proof}
	We verify the claim using the developing map $z:\hat{V} \to \mathbb{C}P^1$, which is determined by $X$ via Equation \eqref{eq:developingmap}. Similarly, every $x \in W^{\mathbb{C}}_X$ determines an infinitesimal change of the developing map $\dot{z}:\hat{V} \to \mathbb{C}$ via the equation for every $ij \in \hat{E}$
		\[
		x_{ij}=  \frac{\dot{z}_i-\dot{z}_k}{z_i-z_k}-\frac{\dot{z}_l-\dot{z}_i}{z_l-z_i}+\frac{\dot{z}_j-\dot{z}_l}{z_j-z_l}-\frac{\dot{z}_k-\dot{z}_j}{z_k-z_j}
		\]
	For every fixed $i$, we consider an infinitesimal deformation of the vertex $i$ with other vertices fixed and also the complex projective structure fixed. We denote $\tilde{x}^{(i)} \in W^{\mathbb{C}}_X$ the corresponding change in the logarithmic cross ratio and lift it to the universal cover such that it is invariant under deck transformations. Recall we number the adjacent vertices of $i$ as $1$, $2$, ..., $r$. For $j=1,2,\dots r$, we have
	\[
	\tilde{x}^{(i)}_{ij}= -\dot{z}_i (\frac{1}{z_{j-1}-z_i}-\frac{1}{z_{j+1}-z_i})
	\]
	and
	\[
		\tilde{x}^{(i)}_{j\,j+1}= -\dot{z}_i (\frac{1}{z_{j+1}-z_i}-\frac{1}{z_{j}-z_i})
	\]
	while zero on other edges. Here the holonomy remains unchanged and hence \[ (\dot{z}\circ\gamma)_{i} = d\rho_{\gamma}(\dot{z}_i).\]
	
	We compare it with $x^{(i)}=h(a^{(i)})$. By construction, we have for $j=1,2,\dots r$
		\[
	x^{(i)}_{ij}= a_{i \,j-1} - a_{i\,j+1} = \frac{(z_r-z_i)(z_1-z_i)}{z_1-z_r} (\frac{1}{z_{j-1}-z_i}-\frac{1}{z_{j+1}-z_i})
	\]
   and
   	\[
   x^{(i)}_{j\,j+1}= -a_{ij}+a_{i \, j+1}=\frac{(z_r-z_i)(z_1-z_i)}{z_1-z_r} (\frac{1}{z_{j+1}-z_i}-\frac{1}{z_{j}-z_i})
   \]
   while zero on other edges. Observe that 
   \[
   \frac{(z_r-z_i)(z_1-z_i)}{z_1-z_r}
   \]
   represents a vector at $z_i$ tangent to the circle through $z_i,z_1,z_r$ and satisfies 
   \[
   \frac{(z\circ\gamma)_r-(z\circ\gamma)_i)((z\circ\gamma)_1-(z\circ\gamma)_i)}{(z\circ\gamma)_1-(z\circ\gamma)_r} =d\rho_{\gamma}( \frac{(z_r-z_i)(z_1-z_i)}{z_1-z_r}).
   \]
   (See \cite[Lemma 4.3]{Lam2017}). Thus we deduce that $x^{(i)}$ is a constant multiple of $\tilde{x}^{(i)}$ and the claim follows.
\end{proof}

The linearization of the cross-ratio equations in fact can be written in terms of the symplectic form.

\begin{corollary}With the symplectic form $\omega^{\mathbb{C}}_{P}$ on the complex vector space  $W^{\mathbb{C}}$,
	\[
	W_{X}^{\mathbb{C}} = \mbox{span}_{\mathbb{C}}\{x^{(1)},x^{(2)},\dots x^{(n)}\}^{\omega} \subset W^{\mathbb{C}}
	\]
	With the symplectic form $\omega_{P}$ on the real vector space $W^{\mathbb{R}}=W^{\mathbb{C}} \cap \mathbb{R}^{E}$,
	\[
		W_{X}^{\mathbb{R}} = \mbox{span}_{\mathbb{R}}\{\Re x^{(1)}, \Im x^{(1)}, \Re x^{(2)}, \Im x^{(2)}, \dots, \Re x^{(n)}, \Im x^{(n)}\}^{\omega} \subset W^{\mathbb{R}}
	\]
\end{corollary}
\begin{proof}
	An element $x \in W_{X}^{\mathbb{C}}$ implies for each $i$, Equation \eqref{eq:crsum} is satisfied, which is equivalent to
	\[
	   0 = 2 \sum_{jk \in E} x_{jk} a^{(i)}_{jk}  = \omega^{\mathbb{C}}_{P}(x^{(i)},x).
	\] 
	The claim for $W^{\mathbb{R}}$ follows similarly.
\end{proof}

With the symplectic complement, the smoothness of $P(\Theta)$ and Conjecture \ref{conj:strong} can be rephrased in terms of the subspace spanned by $x^{(i)}$.

\begin{corollary}
	The Delaunay circle pattern $X \in P(\Theta)$ is infinitesimally rigid if and only if the set \[ \beta:=\{\Re x^{(1)}, \Im x^{(1)}, \Re x^{(2)}, \Im x^{(2)}, \dots, \Re x^{(n)}, \Im x^{(n)}\} \subset W^{\mathbb{R}} \] are linearly independent. 
	
	Secondly, $\omega_P$ is non-degenerate on $W_{X}^{\mathbb{R}}$ if and only if $\omega$ is non-degenerate on the symplectic complement 
	\[
	\mbox{span}_{\mathbb{R}}\{\Re x^{(1)}, \Im x^{(1)}, \Re x^{(2)}, \Im x^{(2)}, \dots, \Re x^{(n)}, \Im x^{(n)}\}
	\]
\end{corollary}
\begin{proof}
	We prove by contradiction. Suppose $X$ is not infinitesimally rigid. Then there exists a non-trivial $x \in W^{\mathbb{R}}_X$ such that $x \in \ker \hol$. It induces an infinitesimal deformation of vertices with the complex projective structure being fixed. Thus there exists $c_1,c_2,\dots, c_n \in \mathbb{C}$ such that
	\[
	x = \sum^n_{i=1} c_i x^{(i)}.
	\]
	It implies $0= \Im(x) = \sum^n_{i=1} \Im(c_i x^{(i)})$ and thus the set $\beta$ is linearly dependent. The converse can be proved by reversing the argument.
	
	It follows from a general fact for symplectic vector spaces that a symplectic form is non-degenerate on a subspace if and only if it is non-degenerate on its symplectic complement (See \cite{Cannas2001}).
\end{proof}

\bibliographystyle{amsplain}
\bibliography{symplectic}

\end{document}